\documentclass{amsart}

\usepackage{amsmath, amsthm, amsfonts, amssymb, amscd}
\usepackage{enumerate}
\usepackage{xcolor}
\usepackage[all]{xy}
\usepackage[colorlinks=blue,citecolor=red]{hyperref} 
\title[On N-Distal Homeomorphisms]
      {On N-Distal Homeomorphisms}

\author[E. Rego]{E. Rego}
\address{Southern University of Science and Technology, Shenzhen-China}
\email{rego@sustech.edu.cn}

\author[J.C. Salcedo]{J.C. Salcedo}
\address{Instituto de Matem\'atica, Universidade Federal do Rio de Janeiro, P. O. Box 68530, 21945-970 Rio de Janeiro, Brazil.}
\email{salcedo@ufrj.br}

\thanks{
J.C.S. was partially supported by CNPq and CAPES from Brazil.
}

\newtheorem{theorem}{Theorem}[section]

\newtheorem{proposition}[theorem]{Proposition}
\newtheorem{lemma}[theorem]{Lemma}

\theoremstyle{definition}
\newtheorem{definition}[theorem]{Definition}
\newtheorem{remark}[theorem]{Remark}
\newtheorem{example}[theorem]{Example}

\newcommand{\de} {\delta}

\newcommand{\la} {\lambda}

\newcommand{\si} {\sigma}       \newcommand{\Si}{\Sigma}

\newcommand{\Z}{\mathbb{Z}}
\newcommand{\N}{\mathbb{N}}
\newcommand{\R}{\mathbb{R}}
\newcommand{\eps}{\varepsilon}

\def\C{\mathbb{C}}
\def\R{\mathbb{R}}
\def\Z{\mathbb{Z}}
\def\N{\mathbb{N}}

\def\Q{\mathbb{Q}}

\def\B{\mathcal{B}}
\def\O{\mathcal{O}}

\DeclareMathAlphabet{\mathpzc}{OT1}{pzc}{m}{it}

\DeclareMathOperator{\diam}{diam}

\begin{document}

\begin{abstract}
We study the $N$-distal property for homeomorphisms on compact metric spaces. For instance, we define $N$-equicontinuity and prove that every $N$-equicontinuous systems are $N$-distal. We introduce the notion of $N$-distal extensions and $N$-distal factors. We also prove that a $M$-distal extension of $N$-distal homeomorphisms is $MN$-distal and present a non-trivial $N$-distal factor for $N$-distal homeomorphisms having Ellis semigroup with a unique minimal ideal. It is also shown that transitive $N$-distal homeomorphisms have at most $N-1$ minimal proper subsystems. Finally, we prove that topological entropy vanishes for $N$-distal systems on compact metric spaces with some nice behavior on the non-wandering set. These results generalize previous ones for distal systems \cite{Furst63},\cite{Parry67}.

\end{abstract}

\maketitle

\section{Introduction}

The distal homeomorphisms were introduced by Hilbert in order to give a topological characterization for the concept of a rigid group of motions (see \cite{Zipp41}). Such homeomorphisms have been widely studied in the literature. For instance, Ellis reduced them to the enveloping semigroups and the minimal distal systems \cite{Ellis58}; Fürstenberg proved a structure theorem \cite{Furst63} and Parry proved that they have zero entropy \cite{Parry67}.

Generalizations of  distal systems include the {\it point distal flows} by Veech \cite{Veech70} who obtained a structure theorem from them and the more recent {\it mean distal systems} by Ornstein and Weiss \cite{OrnWeis04}. From the measure-theoretic viewpoint we can mention Parry's systems with separating sieve also known as {\it measure distal systems} \cite{Parry67}. Zimmer \cite{Zim76} proved a structure theorem for the measure distal systems. Lindenstrauss  proved that any ergodic measure distal system can be realized as a minimal distal system with a fully supported invariant Borel measure \cite{Linden99}. Fürstenberg introduced the notion of a {\it tight system} as one in which, after removing a negligible set, there are no distinct mean proximal points. Ornstein and Weiss also proved  that tight systems have no finite positive entropy \cite{OrnWeis04}.

New classes of systems which generalize the notion of distality were recently introduced by Lee and Morales \cite{ACKM20}, \cite{MorLee17}. They included the notion of $N$-{\it distal}, {\it countably distal}, $cw$-{\it distal} and {\it measurable distal}.

In this paper we study $N$-distal self-homeomorphisms on compact metric spaces, discuss some of their basic dynamical properties and consider the extent to which certain classical results that are already available for distal systems are also valid for $N$-distal systems. Consequently we are interested in studying the relation between $N$-distality and other dynamical properties.  For instance, in connection with the fact that  equicontinuous systems are always distal, we introduce the notion of $N$-equicontinuity (see Definition \ref{defNequi}) and prove a generalization of this result showing that every $N$-equicontinuous systems are $N$-distal in Theorem \ref{Nequi}.

In addition we show how the $N$-distality behaves under homomorphisms. In this way we extend the notion of distal extensions and distal factors defining $N$-distal extensions and $N$-distal factors (see Definition \ref{Ndisext}) and prove in Theorem \ref{DisExt} that a $M$-distal extension of a $N$-distal homeomorphism is $MN$-distal. This result generalizes the previous one in \cite{Aus88}. Using the Ellis semigroups theory we give a criterion for the existence of a non-trivial $N$-distal factor for $N$-distal homeomorphisms in Theorem \ref{NDisFac}.

We further investigate how $N$-distality interacts with topological transitivity. Actually, in Theorem \ref{TranDis} we use the Ellis semigroups theory to obtain a restriction on the number of minimal subsystems for a transitive $N$-distal homeomorphism. This result is a $N$-distal version for the one given in \cite{Aus88}.

Finally, we study the topological entropy of $N$-distal systems. W. Parry \cite{Parry67} showed that  the topological entropy of  a distal system vanishes. One of the key facts for this result is  that the phase space of a distal systems decomposes into a union of minimal subsets \cite{Ellis58}. On the other hand, the same is not valid for $N$-distal systems as we show in Example \ref{ex3dis}. Thus, under a condition that guarantees this kind of decomposition on the non-wandering set, in Theorem \ref{h0}, we prove that $N$-distal homeomorphims have zero topological entropy.

This paper is divided as follows. In section \ref{PreFR} we introduce the main concepts and tools used through out this work as well as some basic properties, examples of $N$-distal systems and generalizations of classical results are established. In section \ref{TranNDS} we investigate how $N$-distality behaves together with topological transitivity and discuss the relation between $N$-distality and expansivity. The final section \ref{TopE} is devoted to the proof of the Theorem \ref{h0}.

%proving the Theorem \ref{TranDis}

\section{Preliminaries and First Results}\label{PreFR}

In this section we introduce the class of $N$-distal systems, study some of its dynamical properties and their connection with other dynamical concepts. First, we recall some basic definitions and state some notation that we will use in this work.

Throughout this paper unless otherwise stated $X$ will denote a compact metric space with metric $d$ and $f:X\to X$ be a homeomorphism. $f^n$ is the $n$-fold self-composition of the map $f$ if $n>0$, on the other hand the $n$-fold composition of $f^{-1}$ if $n<0$ and $f^0$ is the identity map, denoted by $Id$. The orbit of a point $x$ under $f$ is the set $\{f^n(x); n\in \Z\}$ which we denote as ${\O}_{f}(x)$. The closure of $A\subseteq X$ is be denoted by $\bar{A}$ while the cardinality of $A$ is denoted by $\# A$ . If $x\in X$ and $\delta>0$ we denote the {\it open ball} around $x$  by $B_{\delta}(x)$.   

Let $x,y\in X$. Recall that $x$ is said to be {\it proximal} to $y$ if $\inf_{n\in \Z} d(f^n(x),f^n(y))=0$. Clearly, the proximal relation is reflexive, symmetric and invariant, but in general it is not transitive  (see example \ref{exNdis}). The pair $(x,y)$ is a {\it proximal pair} if $x$ is proximal to $y$ and  $(x,y)$ is a {\it distal pair} if it is not a proximal pair. Let us denote by $P(x)$ the set of points $y\in X$ such that $(x,y)$ are proximal pairs, i.e. the {\it proximal cell} (cf. \cite[p. 66]{Aus88}) of $x$.

\begin{equation*}
P(x)=\lbrace y\in X\,\,:\,\,\inf_{n\in \Z} d(f^n(x),f^n(y))=0\rbrace.
\end{equation*}
Notation $P_{f}(x)$ will be used to indicate dependence on $f$ if necessary. Let us recall the definition of distality.

\begin{definition}
A point $x\in X$ is said to be {\it distal point} for $f$ if $P(x)$ reduces to $\{x\}$. Let $Dist(f)$ denote the set of distal points of $f$. We say that $f$ is {\it distal} if $Dist(f)=X$.
\end{definition}

Basic examples of distal homeomorphisms are the identity map and isometries of a metric space. Other non-trivial examples are the equicontinuous homeomorphisms or equivalently uniform almost periodic homeomorphisms (see \cite[p. 36]{Ellis69}). In contrast, distal homeomorphisms are not necessary equicontinuous, a counterexample is the homeomorphism $f:D_{1}\to D_{1}$ defined on the unit disc $D_{1}=\{z\in\C\,:\,\|z\|\leq 1\}$ and given by the formula $f(z)=z\exp(2\pi i\vert z\vert)$. 

Let us recall that a subset $A\subseteq X$ is said to be {\it minimal} if it is closed, $f$-invariant and has no proper closed $f$-invariant subsets. Last condition is equivalent to the orbit of any point in $A$ be dense in $A$. A homeomorphism $f$ is {\it minimal} if $X$ is a minimal set.

Remarkable advances were made when it was established that there are examples of distal and minimal homeomorphism that are not equicontinuous see for instance \cite[Theorem 5.14]{Aus88} or Example \ref{exDMnNE}. The nature these examples led Fürstenberg to his path-breaking structure theorem \cite{Furst63}. 

\subsection{$N$-distal homeomorphisms}\label{DefandPro}

Here we define the $N$-distal homeomorphims, state some of its properties and give a few examples. Recently in \cite{ACKM20} the authors defined the following new classes of systems.

\begin{definition}\label{defNdis}
We say that	$f$ is a $N${\it -distal} (for some $N\in {\N}^{+}$) map if $P(x)$ has at most $N$ points and $f$ is a {\it countably-distal} map if the set $P(x)$ is a countable subset of $X$, for all $x\in X$.
\end{definition}

Our first remark is that  distality clearly implies $N$-distality and $N$-distality clearly implies countably-distality, but the converse do not always hold. For instance, consider the following examples.

\begin{example}\label{ex3dis}
There is a compact metric space $X$ and a $3$-distal homeomorphism $f:X\to X$ which is not $2$-distal.

To see this, let $D=\{(\theta,r)\in\R^2\,:\,1\leq|r|\leq 2\}$ be the annulus in polar coordinates. Define $F:D\to D$ through $F(\theta,r)=(\theta+k\,(\text{mod\,1}),(r-1)^2+1)$ with $k$ an irrational number. Consider $p=(0,\frac{3}{2})$ and $X=S_{1}(\textbf{0})\cup {\O}_{F}(p)\cup S_{2}(\textbf{0})$ where $\textbf{0}=(0,0)$ and $S_{r}(x)$ denotes the circle of radius $r$ centered at $x$. Set $f=F|_{X}$, the restriction of $F$ to $X$. Thus the pairs $(p,(0,1))$ and $(p,(0,2))$ form  proximal pairs for $f$, therefore $f$ is $3$-distal but it is not $2$-distal. 

\end{example} 

If we slightly modify the previous example, we obtain

\begin{example}\label{exNdis}
There are $N$-distal homeomorphisms which are not $N-1$-distal for $N\geq 4$ and there is a countably-distal homeomorphisms which is not $N$-distal for every positive integer $N$.

Let $F:D\to D$ be the function defined in the previous example.
\begin{enumerate}[1.]
\item \label{exNdis1}In $D$ consider $p_n=(0,\frac{1}{n}+1)$ with $2\leq n\leq N-1$ for $N\geq 3$. Define $X=\partial D\cup(\cup_{N-1\geq n\geq 2}\O(p_n))$, where $\partial D$ denotes the boundary of $D$. Let $f$ be given by $f=F|_{X}$. Then $f$ is a $N$-distal homeomorphism which is not $N-1$-distal, since $P_{f}(0,\frac{3}{2})=\left\{(0,1),(0,2)\right\}\cup\left\{p_{n}\,\,:\,\,2\leq n\leq N-1\right\}$.
\item \label{exNdis2} Similarly, consider in $D$ the points $p_n=(0,\frac{1}{n}+1)$ with $n\geq 2$. Define $Y=\partial D\cup(\cup_{n\in\N}\O(p_n))$ and the homeomorphism $g:Y\to Y$ by $g=F|_Y$. Thereby $P_{g}(0,\frac{3}{2})=\left\{(0,1),(0,2)\right\}\cup\left\{p_{n}\,\,:\,\,n\geq 2\right\}$, hence $g$ countably-distal map but it is not $N$-distal for every positive integer $N$. 
\end{enumerate}

\end{example}

\begin{example}\label{ex2dis}
A modification of Example \ref{ex3dis} can be made to obtain a $2$-distal homeomorphism which is not distal. Indeed, the space in this example is the union of two concentric circles and the orbit of the point $p=(0,\frac{3}{2})$ between them. The application of the homeomorphism $f$ makes the points in the circles stay in the circles, while the point $p$ approach the inner circle in future and the outer circle in the past.  If we consider only one circle $S_{1}(\textbf{0})$ (or $S_{2}(\textbf{0})$) a change in the dynamics of the point $p$ to approach the circle $S_{1}(\textbf{0})$ (or $S_{2}(\textbf{0})$) for both future and past, we have the desired example.
\end{example}

The above examples show that these three levels of distality are different.

Before stating the next remark, let us recall  that $A\subset \Z$ is said to be {\it syndetic} if there is $F\subset\Z$ finite such that $\Z = F + A$. We say that a point $x\in X$ is {\it almost periodic} with respect to a homeomorphism $f : X\to X$ if $\lbrace n\in\Z\,\,:\,\,f^{n}(x)\in U\rbrace$ is syndetic for every neighborhood $U$ of $x$ (see \cite{Brin}).  A homeomorphism $f$ is said to be {\it pointwise almost periodic} if every $x\in X$ is almost periodic w.r.t. $f$. Oftenly, almost periodic points are called {\it minimal points}. This is because a point $x$ is almost periodic if and only if the closure of $\O(x)$ is a minimal set, for more details see \cite[Section 2.1]{Brin}. We note in this regard that

\begin{remark}
It is well known that every distal homeomorphism is pointwise almost periodic (see for instance \cite[Proposition 2.7.5]{Brin}). Nevertheless, this is not true in general for $N$-distal homeomorphisms for the example \ref{ex3dis}.
\end{remark} 

Now we state our first result that deals with some elementary properties of $N$-distal homeomorphism. Let us first recall that if $f$ and $g$ are homeomorphisms on compact metric spaces $X$ and $Y$, respectively, the {\it direct product} of $f$ and $g$ is the map $f\times g:X\times Y\to X\times Y$ defined by $(f\times g)(x,y)=(f(x),g(y))$. The product turns out to be a homeomorphism on $X\times Y$ if we equip the space $X\times Y$ with the metric $d^2((x_{1},y_{1}),(x_{2},y_{2}))=\max\lbrace d_{1}(x_{1},y_{1}),d_{2}(x_{2},y_{2})\rbrace$ where $d_{1}$ and $d_{2}$ are the metrics on $X$ and $Y$, respectively.  

\begin{proposition}\label{dispro}
Let $f:X\to X$ and $g:Y\to Y$ be homeomorphisms on $X$ and $Y$ compact metric spaces. The following properties hold:
\begin{enumerate}[(i)]
\item If $f$ is $N$-distal, then $f^k$ is $N$-distal for $k\in\Z$.
\item If $f$ is $N$-distal and $g$ is $M$-distal, then $f\times g$ is $MN$-distal.
\item If $f$ and $g$ are conjugated homeomorphisms, $f$ is $N$-distal if and only if $g$ is $N$-distal. 
\end{enumerate}

\end{proposition}
 
\begin{proof}
Observe that by the definition of proximal cell we have $P_{{f}^k}(x)\subseteq P_{f}(x)$ for every $x\in X$ and therefore $(i)$ follows.

Similarly, $(ii)$ is a consequence of $P_{f\times g}(x,y)\subseteq P_{f}(x)\times P_{g}(y)$ for every $(x,y)\in X\times Y$, and this follows from the definitions of $d^2$ on $X\times Y$ and proximal cell.

Finally to prove $(iii)$, suppose that $h$ is the conjugacy homeomorphism between $f$ and $g$. If $g$ is not $N$-distal. Then, there exists $y\in Y$ such that $P_{g}(y)\setminus\{y\}$ has at least $N$ points. Set $x=h^{-1}(y)$. We claim that $P_{f}(x)\setminus\{x\}$ has at least $N$ points. Indeed, let $p_1,\ldots,p_N$ be distinct points in $P_{g}(y)\setminus\{y\}$. It follows that $d(g^{n_k^i}(y),g^{n_k^i}(p_i))\to 0$ as $k\to \infty$ for $i=1,\ldots,N$. Since $h^{-1}$ is continuous, we have $$d(h^{-1}(g^{n_k^i}(y)),h^{-1}(g^{n_k^i}(p_i)))=d(f^{n_k^i}(h^{-1}(y)),f^{n_k^i}(h^{-1}(p_i)))\to 0.$$ Thus $h^{-1}(p_i)\in P_{f}(x)\setminus \{x\}$ for every $i\in\{1,\ldots,N\}$. Therefore $f$ is not $N$-distal.
\end{proof}

\begin{remark}
The above results are also valid for countably-distal homeomorphisms.
\end{remark}

Now, we are going to investigate the relation between equicontinuity and $N$-distality. Since equicontinuous systems are examples of distal systems, every $N$-distal and equicontinuous homeomorphism  must be distal. 

In order to state a weaker form of equicontinuity we use the concept of $N$-diameter (for some $N\in {\N}^{+}$) defined by the authors in \cite{MeMoVi}, which in turn were inspired by the definition of N-sensitivity given by Xiong in \cite{Xio05}, namely, 

\begin{definition}
Let $X$ be  a compact metric space and $N$ be a positive integer. If $A$ is a subset of $X$, define the {\it $N$-diameter} of $A$ by
\begin{equation*}
\diam_N(A)=\sup\limits_{B\subseteq A}\{\min\limits_{x,y\in B} \{d(x,y)\,:\, x\neq y\}\,:\, B\in \mathcal{C}_{N+1}(X)\}.
\end{equation*}
where $\mathcal{C}_N(X)$ denotes  the set of subsets of $X$ with $N$ elements.
\end{definition}

This concept satisfies the following properties

\begin{lemma}\label{ProNdiam}\cite{MeMoVi} Let $X$ be  a compact metric space and $N$ a positive integer. If $A$,$B$ are subsets of $X$. Then

\begin{enumerate}[(i)]
\item  $\diam_{1}(A) = \diam(A)$, where $\diam(A)$ is the usual diameter of $A$.
\item $\diam_{N}(A) = 0$ if and only if $\#A \leq N$.
\item $diam_{N}(A)\leq diam_{M}(A)$, whenever $M \leq N$.
\item $diam_{N}(A) \leq diam_{N}(B)$, whenever $A \subseteq B$.
\item $diam_{N}(A) = diam_{N}(\bar{A})$.
\item $\frac{diam(A)}{N} \leq diam_{N}(A)$, whenever $A$ is connected. 
\end{enumerate}

\end{lemma}

Also we need the following notation for $a>0$ and $x\in X$
\begin{equation*}
R_{\delta}(x)=\{y\in X\,:\,d(f^i(y),f^i(x))<\delta \text{, for some } i\in \Z\}.
\end{equation*}

Notation $R^{f}_{\delta}(x)$ will be useful to indicate dependence of $f$. Clearly $R_{\delta}(x)$ is an open set and a superset of $B_{\delta}(x)$.\\

We can now define a generalization of equicontinuity. The classical definition of equicontinuity states that a homeomorphism $f$ of a compact metric space $(X, d)$ is said to be {\it equicontinuous} if the family of all iterates of $f$ is an equicontinuous family, i.e., for any $\epsilon>0$, there exists $\delta>0$ such that $d(x, y)<\delta$ implies that $d\left(f^{n}(x), f^{n}(y)\right)<\epsilon$ for all $x,y\in X$ and $n \in \mathbb{Z}$ (cf. \cite[p. 45]{Brin}). To be equicontinuous homeomorphism, it equivalent to say that for every $\eps>0$ there exists $\de>0$ such that $\diam(f^{n}(B_{\de}(x)))<\eps$ for all $x\in X$ and $n\in \Z$. This definition suggests the following one.

\begin{definition}\label{defNequi}
We say that $f$ is $N$-{\it equicontinuous} (for some $N\in {\N}^{+}$) if for every $\eps>0$ there exists $\de>0$ such that $\diam_N(f^{n}(R_{\de}(x)))<\eps$ $\forall x\in X$, $\forall n\in \Z$. 
\end{definition}

The $1$-equicontinuous homeomorphisms are precisely the  equicontinuous ones. In fact, observe that from $N$-equicontinuous definition, $(i)$ and $(iv)$ for $n=1$ in Lemma \ref{ProNdiam} it may be concluded that $1$-equicontinuity implies equicontinuity. Conversely, if $f:X\to X$ is equicontinuous. Let $\eps>0$ and $x\in X$, by equicontinuity, for $y,z\in X$ exists $\de_{1},\de_{2}>0$ such that if $d\left(f^{k_{1}}(x), f^{k_{1}}(y)\right)<\de_{1}$ and $d\left(f^{k_{2}}(x), f^{k_{2}}(z)\right)<\de_{2}$ for some $k_{1},k_{2}\in\Z$, then $d\left(f^{n}(x), f^{n}(y)\right)<\frac{\epsilon}{2}$ and  $d\left(f^{n}(x), f^{n}(z)\right)<\frac{\epsilon}{2}$ for all $n\in\Z$. Set $\de=\min\{\de_{1},\de_{2}\}$. Hence it follows from the triangle inequality and supremum definition that

$$\diam_1(f^{n}(R_{\de}(x)))=\sup\limits_{y,z\in f^{n}(R_{\de}(x))}\{d(y,z)\}=\sup\limits_{y,z\in R_{\de}(x)}\{d\left(f^{n}(y), f^{n}(z)\right)\}<\eps,$$

for all $n\in\Z$. Therefore, $f$ is $1$-equicontinuous.

Clearly, every $N$-equicontinuous homeomorphism is $M$-equicontinuous homeomorphism for all $N\leq M$. It follows that every equicontinuous homeomorphism is $N$-equicontinuous homeomorphism for every $N \geq 1$ and thus such dynamical systems exist. A more subtle problem is to find $N$-equicontinuous homeomorphisms which are not $M$-equicontinuous homeomorphisms for some $M\leq N$. Indeed, this question is answered positively in Example \ref{exNequi}. Before we present this and a related example we first establish the following results.

Notice that a direct computation shows the following identity for any $n\in\Z$, $x\in X$, and $\delta>0$:
\begin{equation}\label{InvofR}
f^{n}(R_{\delta}(x)) = R_{\delta}(f^{n}(x)).
\end{equation}

Using it we obtain the lemma below.

\begin{lemma}\label{NequiAlt}
Let $N\in {\N}^{+}$. The following properties are equivalent for every homeomorphism $f:X\to X$ on a compact metric spaces $X$:

\begin{enumerate}[(i)]
\item  $f$ is $N$-equicontinuous.
\item For every $\eps>0$ there exists $\de>0$ such that $\diam_N(R_{\de}(x))<\eps$ $\forall x\in X$. 
\end{enumerate}

\end{lemma}

\begin{proof}
Clearly $(i)$ implies $(ii)$ so it remains to prove that $(ii)$ implies $(i)$. To see this, let $\eps>0$ and $x\in X$, by $(ii)$ for each $M\in\Z$ there is $\delta(M)$ such that $\diam_N(R_{\de(M)}(f^{M}(x)))<\eps$. Define 

\begin{equation*}
\delta=\inf\limits_{M\in\Z}\{\delta(M)\,:\,\diam_N(R_{\de(M)}(f^{M}(x)))<\eps\,\land\, N<\#R_{\de(M)}(f^{M}(x))\}.
\end{equation*}

Observe that $\delta>0$. Indeed, if $\delta=0$ then there exists a $M_{0}\in\Z$ such that $\#R_{\de(M_{0})}(f^{M_{0}}(x))<N$, a contradiction. Therefore, there is $\delta>0$ such that $\diam_N(R_{\delta}(f^{n}(x)))<\eps$ for all $n\in \Z$. As $x$ and $\eps$ are arbitrary we conclude from (\ref{InvofR}) that $f$ is $N$-equicontinuous proving $(i)$.
\end{proof}

A direct application of the above Lemma is the following result that generalizes the one in \cite{Furst63}.

\begin{theorem}\label{Nequi}
$N$-equicontinuous homeomorphisms are $N$-distal.
\end{theorem}

\begin{proof}
Let $f:X\to X$ be a $N$-equicontinuous homeomorphism of a compact metric space $X$. Suppose by contradiction that $f$ is not $N$-distal. Then, there exists $x\in X$ such that $\#P(x)>N$. Thus let $A\in \mathcal{C}_{N+1}(X)$ be such that $x\in A$ and $A\subseteq P(x)$. Set $\eps_n=\frac{1}{n}$ and let $0<\de_n\leq\eps_n$ be given by the $N$-equicontinuity of $f$ for each $n\in\N$.  Notice that since any $y\in A$ is proximal to $x$, then $A \subseteq R_{\de_{n}}(x)$ for every $n\in {\N}^{+}$. Therefore by Lemma \ref{NequiAlt} $\diam_N(R_{\de_{n}}(x))<\eps_n$ for every $n\in\N$. Since $A$ is finite, then there exists $y\in A$ and a sequence $n_j\to \infty$ such that $d(x,y)<\frac{1}{n_j}$ for every $j\in \N$. It follows that $x=y$ and $\#A\leq N$ which is absurd. Therefore, $f$ must be $N$-distal. 
\end{proof}

We now present some related examples. Combining the above Theorem with examples \ref{ex2dis}, \ref{ex3dis} and \ref{exNdis}.\ref{exNdis1} we obtain the following sentence.

\begin{example}\label{exNequi}
There are $N$-equicontinuous homeomorphisms which are not $N-1$-equicontinuous for every $N\geq 2$.

A simple computation shows that the $N$-distal homeomorphisms in Examples \ref{ex2dis}, \ref{ex3dis} and \ref{exNdis}.\ref{exNdis1} are also $N$-equicontinuous for $N=2$, $N=3$ and $N\geq 4$, respectively. As already verified these examples are not $N-1$-distal homeomorphisms. Consequently,  are not $N-1$-equicontinuous homeomorphims by Theorem \ref{Nequi}.  
\end{example}

The converse direction of the Theorem \ref{Nequi} motivate the question whether there are $N$-distal homeomorphims which are not $N$-equicontinuous. As is well known the answer for $N=1$ is positive, as already mentioned an example can be found in \cite[Theorem 5.14]{Aus88}. The answer also turns out to be positive for $N\geq2$ by the following result. Even more, in the next remarkable Fürstenberg's example we deal with a more general problem on the existence of distal and minimal systems that are not necessarily $N$-equicontinuous.

\begin{example}\label{exDMnNE}
There is a compact metric space $X$ and a minimal distal homeomorphism $f:X\to X$ which is not $N$-equicontinuous for every positive integer $N$.
\end{example}

In order to see this consider the homeomorphism defined by

\begin{equation}
\begin{array}{rccl}
F:&{\mathbb{T}}^2&\longrightarrow&{\mathbb{T}}^2\\
&(x,y)&\mapsto&(x+\alpha\,(\text{mod\,1}),y+x\,(\text{mod\,1}))
\end{array},
\end{equation}

where $\alpha\in(0,1)\setminus\Q$. We view ${\mathbb{T}}^2$ represented as $[0,1)\times[0,1)$, the unit square with opposite sides identified and use the metric inherited from the Euclidean metric. First, note that $F^{n}(x,y)=(x+n\alpha(\text{mod\,1}),y+nx+\sum_{j=0}^{n-1}j\alpha(\text{mod\,1}))$ and $F^{-n}(x,y)=(x-n\alpha(\text{mod\,1}),y-nx+\sum_{j=1}^{n}j\alpha(\text{mod\,1}))$ for all $n\geq 1$. Therefore, the distance between $n$-th iterate by $F$ of two points $(x,y)$ and $(x^{'},y^{'})$ in ${\mathbb{T}}^2$ is
\begin{equation}\label{dofDMnNE}
d(F^{n}(x,y),F^{n}(x^{'},y^{'})) = 
     \begin{cases}
       \sqrt{{(x-x^{'})}^{2}+{(y-y^{'}+n(x-x^{'}))}^2} &\quad\text{if }n\geq 0\\
       \sqrt{{(x-x^{'})}^{2}+{(y-y^{'}+n(x^{'}-x))}^2} &\quad\text{if }n<0\\
     \end{cases},
\end{equation}

where the operations of the terms in parentheses are done in modulo 1. 

We begin by verifying that $F$ is distal. To see this, let $(x_{0},y_{0})$ and $(x_{1},y_{1})$ be distinct points in ${\mathbb{T}}^2$. If $x_{0}\neq x_{1}$, then $d(F^{n}(x_{0},y_{0}),F^{n}(x_{1},y_{1}))\geq d((x_{0},y_{0}),(x_{1},y_{0}))$ which is a positive constant. Analogously, if $x_{0}=x_{1}$, then $y_{0}\neq y_{1}$ and by (\ref{dofDMnNE}) clearly $d(F^{n}(x_{0},y_{0}),F^{n}(x_{1},y_{1}))= d((x_{0},y_{0}),(x_{1},y_{1}))$. In both cases this hold for all $n\in\Z$. Thus, $((x_{0},y_{0}),(x_{1},y_{1}))$ is a distal pair. Since $(x_{0},y_{0})$ and $(x_{1},y_{1})$ are arbitrary we are done. Additionally, $F$ is also minimal, see \cite[Lemma 1.25]{Furst81} for details. 

Finally, fix $N\in{\N}^{+}$, now we prove that $F$ is not $N$-equicontinuous. For this purpose, let $\delta>0$, in ${\mathbb{T}}^2$ consider $q=(0,0)$ and $p_{k}=(\frac{\delta}{k},0)$ for $1\leq k\leq N$. Define $B_{0}=\{q\}\cup\left\{p_{k}\,\,:\,\,1\leq k\leq N\right\}$, then for each $n\in\Z$
\begin{eqnarray*}
\diam_N(F^{n}(R_{\de}(q)))&=&\sup\limits_{\substack{B\in \mathcal{C}_{N+1}(X)\\B\subseteq f^{n}(R_{\de}(q))}}\{\min\limits_{\substack{z\neq w\\z,w\in B}} \{d(z,w)\}\}\\
&\geq &\min\limits_{\substack{z\neq w\\z,w\in B_{0}}} \{d(F^{n}(z),F^{n}(w))\}\\
&=&\sqrt{{\left(\frac{\delta}{M}(\text{mod\,1})\right)}^{2}+{\left(n\frac{\delta}{M}(\text{mod\,1})\right)}^2}\\
&\geq & \vert n\frac{\delta}{M}(\text{mod\,1})\vert,
\end{eqnarray*}

for some $1\leq M \leq N$. Since $\frac{\delta}{M}>0$ we can find $n\in\N^{+}$ such that $\vert n\frac{\delta}{M}\text{mod\,1}\vert\geq \frac{1}{4}$. As $N$ is arbitrary we are done.

Thus, as a consequences of the Theorem \ref{Nequi}, and Definitions \ref{defNdis} and \ref{defNequi} we obtain the following diagram which shows how the concepts of $N$-distality and $N$-equicontinuity interact.
\[
\xymatrix{
\text{Distality} \ar@{=>}[r]\ar@{<=}[d]&N\text{-Distality}\ar@{=>}[r]\ar@{<=}[d]&M\text{-Distality}\ar@{<=}[d]\\ \text{Equicontinuity} \ar@{=>}[r]&N\text{-Equicontinuity}\ar@{=>}[r]&M\text{-Equicontinuity}
}
\]

for all $2\leq N<M$. Moreover, the converse is not true in general for the Examples \ref{ex3dis}, \ref{exNdis}.\ref{exNdis1}, \ref{ex2dis}, \ref{exNequi} and \ref{exDMnNE}.

\begin{remark}
The nature of the example \ref{exDMnNE} led Fürstenberg to his path-breaking structure theorem \cite{Furst63}, describing the structure of a general minimal distal system as a (countable but maybe transfinite) inverse limit of isometric extensions starting with the one-point dynamical system. Indeed, modify the map in the Example \ref{exDMnNE} to $F(x,y)=(x+\alpha,y+2x+\alpha)$. Again show that $(X,F)$ is minimal distal. We can check that $F(0,0)=(n\alpha,n^{2}\alpha)$ and deduce that the sequence ${\{n^2\alpha\}}_{n\in\N}$ is dense in ${\mathbb{S}}^{1}$; see Fürstenberg's book \cite{Furst81} for further development of these ideas. 
\end{remark}

\subsection{Factors and extensions}

Now we study how $N$-distality behaves under factors and extensions. 

Let $g:Y\to Y$ be a homeomorphism of a compact metric space $Y$. Let us first recall that a map $\pi:Y\to X$ is said to be {\it distal} if $\inf\limits_{n\in \Z} d(g^n(y_{1}),g^n(y_{2}))>0$ for every distinct $y_{1},y_{2}\in Y$ satisfying $\pi(y_{1})=\pi(y_{2})$. A {\it homomorphism} from $(Y,g)$ to $(X,f)$ is a continuous onto map $\pi:Y\to X$ satisfying $f\circ\pi=\pi\circ g$. We also  say that $f$ is a {\it factor} of $g$ (or $g$ is an {\it extension} of $f$) under $\pi$. And $g$ is a {\it distal extension} of $f$ (or $f$ is a {\it distal factor} of $g$) if there is a distal homomorphism from $g$ to $f$ (cf. \cite[Definition 4.14]{Ellis14}). The map $F:{\mathbb{T}}^2\to{\mathbb{T}}^2$ in the Example \ref{exDMnNE} is a distal extension of a circle rotation $R_{\alpha}:{\mathbb{S}}^1\to{\mathbb{S}}^1$ given by $R_{\alpha}(x)=x+\alpha\,(\text{mod\,1})$, with projection on the first coordinate as the distal homomorphism.

In order to generalize these notions to the setting of $N$-distal homeomorphisms, we introduce the following auxiliary definition.

\begin{definition}
Let $f:X\to X$ and $g:Y\to Y$ be homeomorphisms of compact metric spaces and $\pi:Y\to X$ a homomorphism from $g$ to $f$. Let $y\in Y$ with $\pi(y)=x$ for some $x\in X$ we define and denote the {\it proximal cell} of $y$ under $\pi$ by
\begin{equation*}
P^{\pi}(y)=\lbrace z\in \pi^{-1}(x)\,\,:\,\,\inf_{n\in \Z}\lbrace d(g^n(y),g^n(z))\rbrace=0\rbrace.
\end{equation*}
Notation $P^{\pi}_{g}(y)$ will be used to indicate dependence on $g$ if necessary.
\end{definition}

We can now define the following concepts 

\begin{definition}\label{Ndisext}
We say that a homomorphism $\pi:Y\to X$ is $N$-{\it distal} if $P^{\pi}(y)$ has at most $N$ points for every $y\in Y$. We say that $g$ is a $N${\it -distal extension} of $f$ (or $f$ is a $N${\it -distal factor} of $g$) if there is a $N$-distal homomorphism $\pi$ from $g$ to $f$.
\end{definition}

As noted in \cite{Aus88} $g$ is a distal extension of the trivial (one point) homeomorphism if and only if $g$ is a distal homeomorphism. Similarly, $g$ is a $N$-distal extension of the trivial homeomorphism if and only if $g$ is a $N$-distal homeomorphism. In this regard it is known that a distal extension of a distal homeomorphism is distal (see \cite[Proposition 5.8]{Aus88}). Next we prove a generalization of this fact.

\begin{theorem}\label{DisExt}
A $M$-distal extension of a $N$-distal homeomorphism is $MN$-distal.
\end{theorem}

\begin{proof}
Let $X$ and $Y$ be compact metric spaces, $g:Y\to Y$ be a $M$-distal extension of a $N$-distal homeomorphism $f:X\to X$, with $\pi:Y\to X$  the $N$-distal homomorphism from $g$ to $f$. Suppose by contradiction that $g$ is not $MN$-distal. Then there is $y\in Y$ such that $P_{g}(y)$ has at lest $MN+1$ elements. Let $p_{1},\ldots,p_{MN},p_{MN+1}=y$ be the different points in $P_{g}(y)$. As a 
consequence of the definition of proximal cell, we obtain $d(g^{n_k^i}(p_i),g^{n_k^i}(y))\to 0$ as $k\to \infty$ for $i=1,\ldots,NM$. Since $\pi$ is continuous and $f\circ\pi=\pi\circ g$, we have 
\begin{equation*}
d(\pi(g^{n_k^i}(p_i))\pi(g^{n_k^i}(y)))=d(f^{n_k^i}(\pi(p_i)),f^{n_k^i}(\pi(y)))\to 0.
\end{equation*}
as $k\to\infty$ for $i=1,\ldots,NM$. Since $\# P_{f}(\pi(y))\leq N$, it follows that there are $p_{l_{1}},\ldots,p_{l_{M+1}}$ different points such that $p_{l_{1}},\ldots,p_{l_{M}}\in\pi^{-1}(p_{l_{M+1}})$. Thus $p_{l_{1}},\ldots,p_{l_{M}}\in P_{g}^{\pi}(p_{l_{M+1}})$ and  therefore there is a point $z=p_{l_{M+1}}$ in $Y$ such that $\# P_{g}^{\pi}(z)>M$, a contradiction.
\end{proof}

We end this section by dealing with the problem of determining when a $N$-distal system has a non-trivial distal factor. 

One of the most useful tools to study topological dynamics is the Ellis semigroup of a transformation was introduced by R. Ellis in \cite[Definition 8]{EllisGot60}. Let us briefly introduce this notion and some interesting facts about it.

Let $X$ be a compact metric space and denote  $X^X$ for the set of all self-transformations of $X$ (continuous or not). So $X^X$ is a compact topological space by Tychonoff's Theorem and  its topology can be seen as the topology of the pointwise convergence. We can put a  semigroup structure on $X^X$ considering its the composition operation. Now let $f:X\to X$ be a homeomorphism. The {\it Ellis semigroup} $E(f)$ of $f$ is the closure of the set $\{f^n\,\,:\,\,n\in \Z\}$ in $X^X$; for more details we refer the reader to \cite{Ellis14}.

A interesting fact about the semigroup $E(f)$ is that one can translate algebraic properties of $E(f)$ in to dynamical properties of $f$. For instance, idempotent elements and minimal ideals are related to minimal sets for $f$. Next, we use these semigroups to obtain a criterion for existence of non-trivial $N$-distal factors for $N$-distal homeomorphisms.      

\begin{theorem}\label{NDisFac}
Let $f$ be a $N$-distal homeomorphism. If the Ellis semigroup $E(f)$ of $f$ has a unique minimal ideal, then $f$ has a non-trivial $N$-distal factor. Moreover, this factor is distal.
\end{theorem}

\begin{proof}
Suppose that $E(f)$ has a unique minimal ideal. It is a classical fact that this condition is equivalent to the proximal relation $"\sim"$ in $X$ be an equivalence relation (cf. \cite[Theorem 2]{Ellis60}). Then define $Y=X/\sim$ to be the quotient space of $X$ by proximality relation and let $\pi$ denote the natural projection map.  Let $g$ be the homeomorphism induced on $Y$ by $f$ through the projection $\pi$. We notice that the conjugacy equation is trivially satisfied for $g$ and $f$. It follows from the $N$-distality of $f$ that $P^{\pi}_{f}(y)$ has at most $N$ points for all $y\in Y$. Then $g$ is a $N$-distal factor of $f$.

Next we prove that the  homeomorphism  $g$  is  distal. Indeed, suppose that $y,y'\in Y$ are distinct  proximal points for $g$. Let us take $x\in\pi^{-1}(y)$ and $x'\in\pi^{-1}(y')$. By construction $x$ and $x'$ are distal. Compactness of $X$ implies that there are a sequence of $k\to \infty$ and a point $z\in Y$ such that $g^{k}(y), g^{k}(y') \to z$. We can assume by compactness of $X$ that there  exists $p,p'\in X$ such that  $f^{k}(x)\to p$ a and $f^{k}(x')\to p'$. 

We claim that $p$ and $p'$ are distal. Indeed, suppose that there is $i\to \infty$ and $z'$ such that $f^i(p),f^i(p')\to z'$. Fix  $\eps>0$ and $i_0$ such that $f^{i_0}(p),f^{i_0}(p')\in B{\eps}(z')$. So there exists $\de>0$ such that if $d(u,w)\leq \de$ then $d(f^j(u),f^j(w))<\eps$ for $j=0,..., i_0$ and every $u,w\in X$. Tale $k$ big enough such that $ f^k(y)\in B_{\de}(p) $ and $ f^k(y')\in B_{\de}(p')$. But this implies $d(f^{k+i_0}(x),f^{k+i_0}(x'))\leq 4\eps$. Remember $x$ and $x'$ are distal and therefore it is impossible since $\eps>0$ was chosen arbitrarily. Thus $p$ and $p'$ must be distal.

Finally we must have $\pi(p)=\pi(p')=z$ by continuity of $\pi$, but this is impossible since $p$ and $p'$ cannot be in the same equivalent class.   
\end{proof}

\section{Transitive $N$-distal Systems}\label{TranNDS}

In this section we study some consequences of the topological transitivity for $N$-distal systems and discuss the relation between $N$-distality and expansivity. First recall that a homeomorphism is {\it transitive} if for any pair of non-empty open sets $U$ and $V$ one can find an integer $n$ such that $f^{n}(U)\cap V$ is nonempty. We say that a point $x$ is a {\it transitive point} of $f$ if its orbit is dense on $X$. Every point of minimal homeomorphism is a transitive point. We say that a $f$ is {\it pointwise transitive} if there exists some transitive point for  $f$.  

We remark that for second countable spaces and in absence of isolated points, point transitivity is equivalent to topological transitivity; for more details we refer the reader to \cite{AoHi94}.

Next results deals with the existence of periodic orbits. But previously we need the following proposition.

\begin{proposition}\label{perp}
Let $f: X\to X$ be a $N$-distal homeomorphism for $N\geq 2$. If $x\in X$ is periodic for $f$, then  $x$ is a distal point.
\end{proposition}

\begin{proof}
Let $x$ be a periodic point of $f$ with period $T$. If $P(x)\neq\{x\}$, take $y\in P(x)\setminus\{x\}$. Then, there is a sequence $n_k\to \infty$ such that $d(f^{n_k}(x),f^{n_k}(y))\to 0$. Since $x$ is periodic and $y\neq x$ then $y$ cannot be periodic. Moreover, since the orbit of  $x$ is finite, we can assume that $f^{n_k}(x)=p$ for any $k\in \N$ and some $p\in \O(x)$. Last assumption implies that $n_k-n_{k'}$ is a multiple of $T$ for any $k,k'\in \Z$.  For any $k\in \N$ we set $y_k=f^{n_k}(y)$.

We claim that $y_k$ is proximal to $p$ for every $k$. Indeed, fix $k$ and define $m^k_j=n_j-n_k$ for $j\geq k$.
Then we obtain that $d(f^{m^k_j}(y_k),f^{m^k_j}(p))=d(f^{n_j}(y),f^{n_j}(p))=d(f^{n_j}(y),p)\to 0$ and this proves the claiming.

Finally, notice that $y$ cannot be a periodic. Thus, we have infinitely many $y_k$'s and therefore $\#P(p)=\infty$, a contradiction.    
\end{proof}

As a consequence we show that the only way a  transitive $N$-distal system can possess a periodic orbit is if the whole space is a periodic orbit.

\begin{proposition}
Let $f: X\to X$ be a pointwise transitive $N$-distal homeomorphism which is not distal. Then either $X$ is a periodic orbit, or $f$ has not periodic points. 
\end{proposition}

\begin{proof}
Suppose that $X$ is not a periodic orbit and  let $p$ be a periodic point. Suppose $x\in X$ is a transitive point.  Since $\O(x)$ is dense, then for any point of $q\in \O(p)$ we can find a sequence $n_k\to \infty$ such that $f_{n_k}(x)\to q$.  Let $T$ denote the period of $p$. 
Since $f$ is continuous, for every $k\in \N$ we can find $0<\de_k< \frac{1}{k}$ such that if $d(x,y)\leq \de_k$ then $d(f^i(x),f^i(y))\leq\frac{1}{k}$ for  $|i|\leq T$.  Up to take a subsequence of $n_k$, we can suppose that $ f^{n_k}(x)\in B_{\de_k}(p)$ for any $k\in \N$. 

Now, by the choice of $\de_k$ we have that $f^i(x_k)\in B_{\frac{1}{k}}(f^i(p))$ for $i=0,1,...,T$. By the euclidean algorithm any $n_k$ can be wrote as $n_k=q_kT+r_k$ with $q_k\in \N$ and $0<r_k< T$. Since the orbit of $p$ is finite, we can assume that $r_k=c$ for every $k$.  Put   $x_k=f^{n_k}(x)$ for every $k\in \N$.

We claim that the points $x_k$ are proximal to $p$. Indeed, fix $k$ and for any $j>k$ define $m^k_j=n_j-n_k+T=T(q_j-q_k+1)$. Then we have that $d(f^{m^k_j}(x),f^{m^k_j}(p))=d(f^{n_j-n_k+T}(x_k),f^{n_j-n_k+T}(p'))=d(f^{n_j}(x),f^{T(q_j-q_k+1)}(p))=d(f^{n_j+T}(x),p)\leq \frac{1}{j}$.

Since this, in addiction to the fact that $\{x_k\}$ is an infinite set implies that $\#P(p)=\infty$, we obtain a contradiction by Proposition \ref{perp}. This proves the result.
\end{proof}

As it is well known, a distal homeomorphism is pointwise transitive if and only if is minimal (cf. \cite[Corollary 5.7]{Aus88}). This is not valid for $N$-distal homeomorphism by the Example \ref{ex3dis}. Nevertheless, the transitive $3$-distal homeomorphism given in this example has two minimal subsystems. It is then natural to ask if there is a relation between the number of minimal subsystems of a dynamical system and the transitive $N$-distal property. Indeed, the Theorem \ref{TranDis} gives an answer for this question. Recall that an idempotent element $g$ of a semigroup $G$ is an element satisfying $g^2=g$.

The following lemma is proved as in \cite[Lemma 2]{Aus60}.

\begin{lemma}\label{LemEllis}
Let $f: X\to X$ be  a homeomorphism on a compact metric space $X$ and $x\in X$. If $A\subset \overline{\O(x)}$ be a minimal set, then there is $y\in A$ such that $y\in P(x)$.
\end{lemma}

\begin{proof}
Let $E(f)=E$ be the Ellis semigroup of $f$. We claim that $H=\{h\in E\,\,:\,\,h(x)\in A\}$ is a minimal left ideal of $E$. Indeed, let $g\in E$ and $h\in H$. Then, there is a sequence $n_k\to \infty$ such that $f^{n_k}\to g$. It follows that
\begin{equation*}
f^{n_{k}}(h(x))\to g(h(x))=(g\circ h)(x).
\end{equation*}
Since $A$ is $f$-invariant, $f^{n_{k}}(h(x))\in A$. As $A$ is closed we have $(g\circ h)(x)\in A$. So $g\circ h\in H$. Hence, $H$ is a left ideal. Moreover, the minimality condition for $H$ follows from that of $A$. The claim is proved.

We have from \cite[Remark 4]{Ellis60} and the claim that there is an idempotent element $k$ in $H$. Let $k(x)=y$, then $k(y)=k^{2}(x)=k(x)$. Therefore $y\in P(x)$ by \cite[Remark 6]{Ellis60}, which completes the proof.
\end{proof}

Using this lemma we obtain the proposition below.

\begin{proposition}\label{descom}
If $f$ is a $N$-distal homeomorphism, then $f|_{\overline{\O(x)}}$ has at most $N-1$ proper minimal subsystems. 
\end{proposition}

\begin{proof}
First notice that two minimal subsets $A,B\subset X$ have non-empty intersection then they must be equal. Now, suppose $f$ is $N$-distal and fix $x\in X$. Let us analyze the subsystem $f|_{\overline{\O(x)}}$. If $f$ is distal then $\overline{\O(x)}$ is minimal \cite[Corollary 5.7]{Aus88} or $f$ is minimal, in both cases we are done. If it is not minimal, it is well known that there exists a non-trivial minimal subset $A\subset \overline{\O(x)}$. By Lemma \ref{LemEllis} we have that there exists a $y\in A$ proximal to $x$. Clearly $x\neq y$. The latter fact is valid for any minimal subset of $\overline{\O(x)}$. Thus $N$-distality implies that there are at most $N-1$ minimal subsets on $\overline{\O(x)}$ and the proposition is proved.
\end{proof}

A direct consequence of the preceding Proposition is the following $N$-distal version for the one given in \cite{Aus88}. Indeed, we just need to notice that if a system is transitive  there exists some point $x\in X$ such that $\overline{\O(x)}=X$.

\begin{theorem}\label{TranDis}
A transitive $N$-distal homeomorphism has at most $N-1$ minimal proper subsystems.
\end{theorem}

Now we proceed to study the relation between $N$-distality and expansivity. It  is a classical result that a distal system cannot be expansive if the phase is sufficiently rich.  Indeed, this result is a consequence of  \cite[Theorem 2.1]{Utz50} and can be found in \cite{Brin}. Keeping this in mind, we ask if the same is true for the weaker forms of distality and expansiveness. Actually, it is answered by the authors in \cite{ACKM20} when the phase space has positive topological dimension. Before stating this result precisely, let us recall some definitions. 

The $\delta$-{\it dynamical ball} centered at $x$ is the set defined by
\begin{equation*}
\Gamma_{\delta}(x)=\{y\in X\,:\, d(f^n(x),f^n(y))<\delta, \forall n\in \Z\}.
\end{equation*}

And recall that a {\it continuum} is a compact and connected set. We say that a continuum is trivial if it is a singleton. 

We can now define the notions of expansiveness, $cw$-expansiveness, $n$-expansiveness and countably-expansiveness were defined in \cite{Utz50}, \cite{Kato93}, \cite{Mor12} and \cite{MorSir13} respectively. Namely

\begin{definition}\label{defexp}
We say that
\begin{enumerate}[1.]
\item $f$ is {\it expansive} if there exists $\delta>0$ such that $\Gamma_{\delta}(x)=\{x\}$ for every $x\in X$.
\item $f$ is $N${\it- expansive} if there exists $\delta>0$ such that $\#\Gamma_{\delta}(x)\leq N$ for every $x\in X$.
\item $f$ is {\it countably-expansive} if there exists $\delta>0$ such that $\Gamma_{\delta}(x)$ is countable for every $x\in X$.
\item \label{cwexp}$f$ is {\it $cw$-expansive} if there exists $\delta>0$ such that if a  non-empty continuum $C\subseteq X$ satisfies $\diam(f^n(C))<\delta$ for every $n\in \Z$, then $C$ is a singleton.
\end{enumerate}
\end{definition}

Clearly, expansive implies $N$-expansive and $N$-expansivity  implies countably-expansivity. The following proposition completes this hierarchy.  

\begin{proposition}\label{N cw}
Any  countably-expansive homeomorphism is  $cw$-expansive.
\end{proposition}
\begin{proof}
Let $\delta>0$ be the countably-expansivity constant of $f$. If a continuum $C$ satisfies $\diam(f^n(C))<\delta$ for every $n\in \Z$, then $C\subset \Gamma_{\delta}(x)$ for any $x\in C$. Thus, $C$ must to be countable, and this implies that $C$ is a singleton.
\end{proof}

Previous proposition allows us to classify the levels of expansiveness accordingly the following  hierarchy.
\begin{equation*}
Expansivity \Rightarrow N-Expansivity \Rightarrow Countably-Expansivity \Rightarrow cw-Expansivity.
\end{equation*}

In contrast, the converse is not true in general(see for instance \cite{Mor12} and \cite{Kato93} for examples). 

On the other hand. In \cite{ACKM20} the authors incorporated distal systems into the Continuum Theory in the same way that Kato in \cite{Kato93}. They noted that the definition \ref{defexp}.\ref{cwexp} holds if we replace $\diam\left(f^{n}(C)\right)$ by $\sup _{n \in \Z}\diam\left(f^{n}(C)\right)$. Also they noticed that to be distal homeomorphism is then equivalent to say that 
\begin{equation*}
\text{ if } C \in 2^{X} \text{ and } \inf _{n \in \Z} \operatorname{diam} f^{n}(C)=0 \Rightarrow \diam(C)=0.
\end{equation*}
By restricting $C$ to the class of nonempty continuum subsets of $X$ the authors \cite{ACKM20} introduced the notion of continuum-wise distal homeomorphisms.

\begin{definition}
A homeomorphism is $cw$-distal if the only continuums $C\subset X$  satisfying $\inf_{n\in \Z} diam(f^n(C))=0$ are singletons.
\end{definition} 

It is easy to see that for distality we have the following hierarchy.
\begin{equation*}
Distality \Rightarrow N-Distality \Rightarrow Countably-Distality \Rightarrow cw-Distality.
\end{equation*}

For the examples \ref{exNdis}.\ref{exNdis2} and \ref{ex2dis} the converse of the first two implications is false in general. Also, the converse of the last implication do not always hold, see for instance \cite[Example 1.7]{ACKM20}.

Next results gives us a distinction between all these levels of distality and expansiveness.

\begin{theorem}\cite{ACKM20}
A $cw$-expansive homeomorphism of a compact metric space of positive topological dimension cannot be $cw$-distal.
\end{theorem}

Despite  above result, we cannot distinguish between $cw$-expansivity and $cw$-distality in zero dimension. Indeed, by definition any system  in a totally disconnected space is $cw$-distal and $cw$ expansive. But as we see in the following example, as an application of the Theorem \ref{TranDis}, we cannot say the same for $N$-distal and $N$-expansive.

\begin{example}
There is $cw$-distal expansive homeomorphism which is not $N$-distal for every positive integer $N$.
\end{example}
Let $\Si^2=\{0,1\}^{\Z}$ with the metric $d(s,s')=\frac{1}{2^n}$ where $n=\min\{|n|\,:\,s_n\neq s'_n\}$ and $d(s,s')=0$ if $s_i=s'_i$ $\forall i\in \Z$. The shift map   $\si:\Si^2\to \Si^2$ is defined by $\si((s_i))=(s_{i+1})$. It is well known that the shift is a transitive and expansive system. Since $\Si^2$ is totally disconnected it is also $cw$-distal. On the other hand,  $\si$ has infinitely many periodic orbits and therefore it cannot be $N$-distal by Theorem \ref{TranDis}.

\section{Topological Entropy}\label{TopE}

In \cite{Parry67} W. Parry proved that distal homeomorphisms on a compact metric space have zero entropy. The purpose of this section is  to extend this result to $N$-distal homeomorphisms.

To start with let $\mathcal{P}=\{P_1,P_{2},...\}$ be a countable partition of a probability space $(X,\B,\mu)$. The {\it entropy of the countable partition} $\mathcal{P}$ is defined by
$$H_{\mu}(\mathcal{P})=-\sum_{P\in \mathcal{P}}\mu(P)\log(\mu(P)).$$

Let $f: X\to X$ be a measurable map preserving a probability measure $\mu$ (i.e. $\mu(f^{-1}(B))=\mu(B)$ for every measurable set $B$), we also say that $\mu$ is $f$-invariant. The {\it metric entropy} of $f$ with respect to $\mu$  is defined by
$$h_{\mu}(f)=\sup \{ h_{\mu}(f,\mathcal{P})\,:\, \mathcal{P}\mbox{ is countable partition of } X \mbox{ with } H_{\mu}(\mathcal{P})<\infty\},$$ 

where    
 $$h_{\mu}(f,\mathcal{P})=\lim_{n\to\infty}\frac{1}{n} H_{\mu}( \mathcal{P}^n)$$
and $\mathcal{P}^n=\{ P_{i_{1}}\cap f^{-1}(P_{i_2})\cap \cdots \cap f^{-n+1}(P_{i_n})\,:\, P_{i_j}\in \mathcal{P},\,j=1,...,n \}$.

On the other hand, topological entropy is defined in a similar way but in topological terms. The variational principle tells us that topological entropy is achieved by the supremum of the metric entropies (see \cite{WAL}). Moreover, this supremum can be taken on  the metric entropies of ergodic measures (see for instance \cite[Corollary 8.6.1]{WAL}). Then we can define {\it topological entropy} of a continuous map $f: X\to X$ in a compact metric space as follows.
$$h(f)=\sup h_{\mu}(f),$$
where the supremum is taken over the set of ergodic $f$-invariant measures of $f$. In particular, it follows that the topological entropy of $f$ is zero if and only if $h_{\mu}(f)=0$ for every ergodic $f$-invariant measures $\mu$.

Recall that a measure $\mu$ is {\it ergodic} if every $f$-invariant measurable set (i.e. $f^{-1}(B)=B$ for every measurable set $B$) has either total measure or null measure.

Also the following elementary facts will be useful later one.

Let $(X,\B,\mu)$ be a probability space. A set $A\in\mathcal{B}$ is called an {\itshape atom} of the measure $\mu$ if $\mu(A)>0$ and every measurable set $E\subset A$ has measure either $0$ or $\mu(A)$. We say that $\mu$ is {\itshape non-atomic} if it has not atoms. Clearly a non-atomic measure has no points of positive mass, and conversely for regular Borel probability measures on compact Hausdorff spaces. Further a non-atomic measure has no finite sets of positive measure.

Given subsets ${\xi}_{1}$ and $\xi_{2}$ of the $\sigma$-algebra $\B$. We write $\xi_{1}\subseteq\xi_{2}\mod\mu$, if for every $E_{1}\in\xi_{1}$ there exists $E_{2}\in\xi_{2}$ such that $\mu( E_{1}\bigtriangleup E_{2})=0$. ($A\bigtriangleup B$ denotes the symmetric difference of the sets $A,B$.) %(kreley pag 435 english version) 
Thereby, we write $\xi_{1}=\xi_{2}\mod\mu$, if both inclusions hold$\mod\mu$. Consequently, for non atomic measures ${\epsilon}_{N}=\epsilon\mod\mu$, where $\epsilon$ and ${\epsilon}_{N}$ are the partition of $X$ into singletons and a partition of $X$ into sets with $N$ or less elements, respectively.

We are in position to prove the main result of this section in which $AP(f)$ and $\Omega(f)$ denotes the set of almost periodic points and the set of non-wandering points of $f$, respectively.

\begin{theorem}\label{h0}
Let $f:X\rightarrow X$ be a $N$-distal homeomorphism on a compact metric space $X$. If $\Omega(f)\subseteq AP(f)$, then $f$ has zero entropy. 
\end{theorem} 

\begin{proof}

We first prove the Theorem for minimal systems and later we will show how to remove this hypothesis. 

By the previous discussion it is sufficient to prove that the metric entropy of $f$ is zero for all ergodic $f$-invariant measure. 
Let $\mu$ be an ergodic $f$-invariant measure. If $\mu$ is atomic then it must be supported on a periodic orbit or a fixed point and therefore its entropy is null, thus  we will assume $\mu$ is non-atomic.

By \cite[Corollary 1.12.10]{Boga} we can fix $0<r<\frac{1}{e}$ and chose a sequence of open sets 

\begin{equation*}
X=S_{0}\supseteq S_{1}\supseteq\cdots\supseteq S_{n}\supseteq\cdots
\end{equation*}
Such that $\mu(S_{n})\leq r^n\,\,\,\forall n\geq 0$ and $\bigcap\limits_{n=0}^{\infty}S_{n}=\lbrace z \rbrace$ for some $z\in X$. Define $\xi=\lbrace E_{0},E_{1},\ldots\rbrace$, where
\begin{equation}\label{defofE}
E_{0}=\lbrace z\rbrace\cup\left(S_{0}\setminus S_{1}\right)\mbox{\hspace{1cm}and\hspace{1cm}}  E_{n}=S_{n}\setminus S_{n+1}\,\,\,\forall n\geq 1.
\end{equation}

Clearly $\xi$ is a partition of $X$. Moreover, since the function $-x\log(x)$ is increasing on $(0,\frac{1}{e})$. It follows that

\begin{equation}\label{KSH1}
H_{\mu}(\xi)\leq-\sum\limits_{n}\mu\left(S_{n}\right)\log\left(\mu\left(S_{n}\right)\right)\leq-\sum\limits_{n}e^{-n}\log\left(e^{-n})\right)=\sum\limits_{n}ne^{-n}=\dfrac{e}{{(e-1)}^{2}}<\infty.
\end{equation}

Now, as we are in the minimal case. We claim that $\bigvee\limits_{j=0}^{\infty}f^{-j}(\xi)={\epsilon}_{N}$, where ${\epsilon}_{N}$ is the partition of $X$ in sets with $N$ or less elements. Indeed, take $B\in\bigvee\limits_{j=0}^{\infty}f^{-j}(\xi)={\epsilon}_{N}$ and $x\in B$. If $y\in B$, then $f^{n}(x),f^{n}(y)\in E_{i_{n}}$ for the same sequence. In fact, since

\begin{equation*}
\bigvee\limits_{j=0}^{\infty}f^{-j}(\xi)=\lbrace E_{i_{0}}\cap f^{-1}(E_{i_{1}})\cap\cdots\cap f^{-n+1}(E_{i_{n-1}})\cap\cdots\,\,:\,\,E_{i_{j}}\in \xi\rbrace,
\end{equation*}

it follows that $x,y\in B\subseteq f^{-n}(E_{i_{n}})$ for some sequence $\left(i_{0},i_{1},\dots\right)$. 

By the minimality of $f$, for each $m\in\N$ there is $k(m)$ such that $f^{k(m)}(x)\in S_{m}$. Then $f^{k(m)}\in E_{i_{k(m)}}\subseteq S_{m}$ for all $i_{k(m)}\geq m$, by (\ref{defofE}). We have $\diam(S_{n})\rightarrow 0$ when $n\rightarrow\infty$ because $S_{n}\rightarrow\lbrace z\rbrace$ when $n\rightarrow\infty$, therefore 
\begin{equation*}
\inf\limits_{n}\lbrace f^{n}(x),f^{n}(y)\rbrace=0.
\end{equation*}
Hence $y\in P(x)$ and then $\#(B)\leq \#(P(x))\leq N$, by $N$-distality of $f$. The claim is proved. Furthermore, as previously remarked, since $\mu$ is non-atomic, ${\epsilon}_{N}=\epsilon\mod\mu$. Thus, $\xi$ is a generating partition (cf. \cite[Definition 14.5]{Sinai76}).

Then, (\ref{KSH1}) and the claim implies that we can use the Kolmogorov-Sinai Theorem \cite[Theorem 14.3]{Sinai76} to obtain
%$\xi$ satisfies the hypotheses of Kolmogorov-Sinai theorem \cite[Theorem 14.3]{Sinai76}. Indeed,
 
\begin{eqnarray*}
h_{\mu}(f)=h_{\mu}(f,\xi)&\leq &H_{\mu}(\xi)\\
&=&-\sum\limits_{n=0}^{\infty}\mu(E_{n})\log(\mu(E_{n}))\\
&=&-\mu(E_{0})\log(\mu(E_{0}))-\sum\limits_{n=1}^{\infty}\mu(E_{n})\log(\mu(E_{n}))\\
&\leq &-\log(\mu(E_{0}))-\sum\limits_{n=1}^{\infty}r^{n}\log(r^{n})\\
&=&-\log(\mu(E_{0}))-\sum\limits_{n=1}^{\infty}nr^{n}\log(r)\\
&=&-\log(\mu(E_{0}))-\frac{r}{{\left(1-r\right)}^{2}}\log(r)\\
&\leq &\log\left(\frac{1-r}{1-2r}\right)-\frac{1}{{\left(1-r\right)}^{2}}r\log(r).
\end{eqnarray*}

As the last expression converges to zero when $r$ goes to zero, we have $h(f)=0$.

\vspace{0.1in}

We now deal with the general case. If $f$ is not minimal. Let $\nu$ be an $f$-invariant measure. It is well-known that $h_{\nu}(f)=h_{\nu}(f|_{\Omega}(f))$. Since $\Omega(f)\subseteq AP(f)$, it follows from \cite[Corollary 1.10]{Aus88} that $\Omega(f)=\cup M_{\la}$, where in the previous (necessarily disjoint) union each $M_{\la}$ is a minimal subset. Moreover, $\eta=\{M_{\la}\}$ is a measurable partition of $\Omega(F)$ by minimality of each $M_{\la}$ and second countable property of $X$. Then  there exists a family of measures $\{\nu_{\la}\}$ decomposing $\nu$. According to the above case, we have $h_{\nu_{\la}}(f|_{M_{\la}})=0$. Therefore,

$$ h_{\nu}(f|_{\Omega(f)})=\int_{\Omega(f)_\eta}h_{\nu_{\la}}(f|_{M_{\la}})d\nu_{\eta}=0,$$             

where $\Omega(f)_\eta$ denotes the factor space of $X$ with respect to $\eta$, and $\nu_{\eta}$ is the factor measure on $\Omega(f)_\eta$. This completes the proof.
\end{proof}

\textit{Acknowledgments: The authors would like to thank C.A. Morales for his great help and his lectures on topological dynamics during the second semester of 2017 at UFRJ which inspired this work.}

\end{document}